\documentclass[11pt]{amsart}

\usepackage{amsmath}
\usepackage{amssymb}
\theoremstyle{plain}
\newtheorem{lemma}{Lemma}[section]
\newtheorem{proposition}[lemma]{Proposition}
\newtheorem{theorem}[lemma]{Theorem}

\newtheorem{corollary}[lemma]{Corollary}

\theoremstyle{definition}
\newtheorem{definition}[lemma]{Definition}
\newtheorem{remark}[lemma]{Remark}
\newtheorem{example}[lemma]{Example}

\usepackage{graphicx,xy,pstricks}
\input xy
\xyoption{all}

\newcommand{\executeiffilenewer}[3]{%
 \ifnum\pdfstrcmp{\pdffilemoddate{#1}}%
 {\pdffilemoddate{#2}}>0%
 {\immediate\write18{#3}}\fi%
}

\newcommand{\Q}{\mathbb{Q}}

\newcommand{\Z}{\mathbb{Z}}
\newcommand{\G}{\mathbb{G}}
\newcommand{\C}{\mathbb{C}}
\newcommand{\N}{\mathbb{N}}

\newcommand{\sldeux}{\mathrm{SL}_2}

\renewcommand{\epsilon}{\varepsilon}

\renewcommand{\phi}{\varphi}

\DeclareMathOperator{\hm}{Hom}
\DeclareMathOperator{\tr}{Tr}
\DeclareMathOperator{\irr}{irr}
\DeclareMathOperator{\ab}{ab}
\DeclareMathOperator{\ad}{Ad}
\DeclareMathOperator{\alg}{alg}
\DeclareMathOperator{\rk}{rank}
\DeclareMathOperator{\redu}{red}
\DeclareMathOperator{\abel}{ab}
\DeclareMathOperator{\ir}{irr}
\DeclareMathOperator{\cent}{cen}
\DeclareMathOperator{\spa}{Span}
\DeclareMathOperator{\spec}{Spec}

\newcommand{\lie}{\mathrm{sl}_2}

\title{% fill in the title
Character varieties in $\sldeux$ and Kauffman skein algebras}

\author{% fill in your name
Julien March\'e}

\address{% fill in your affiliation
Institut de Math\'{e}matiques\\
Universit\'{e} Pierre et Marie Curie\\
75252 Paris c\'{e}dex 05\\
France}

\email{% fill in your e-mail address
julien.marche@imj-prg.fr}

\begin{document}

\maketitle
%\title{Character varieties in $\sldeux$ and skein algebras}
%\date{}
%\author{Julien March\'e}
%\address{Institut de Math\'{e}matiques\\
%Universit\'{e} Pierre et Marie Curie\\
%75252 Paris c\'{e}dex 05\\
%France}
%\email{julien.marche@imj-prg.fr, }
%\begin{document}
%\maketitle
\begin{abstract}
These lecture notes concern the algebraic geometry of the character variety of a finitely generated group in SL$_2(\C)$ from the point of view of skein modules. We focus on the case of surface and 3-manifolds groups and construct the Reidemeister torsion as a rational volume form on the character variety. 
\end{abstract}
\section{Introduction}
These notes collect some general facts about character varieties of finitely generated groups in $\sldeux$. We stress that one can study character varieties with the point of view of skein modules: using a theorem of K. Saito, we recover some standard results of character varieties, including a construction of a so-called tautological representation. This allows to give a global construction of the Reidemeister torsion, which should be useful for further study such as its singularities or differential equations it should satisfy. 

The main motivation of the author is to understand the relation between character varieties and topological quantum field theory (TQFT) with gauge group SU$_2$. This theory - in the \cite{bhmv} version - makes fundamental use of the Kauffman bracket skein module, an object intimately related to character varieties. Moreover the non-abelian Reidemeister torsion plays a fundamental role in the Witten asymptotic expansion conjecture, governing the asymptotics of quantum invariants of 3-manifolds. 
However, these notes do not deal with TQFT, and strictly speaking do not contain any new result. Let us describe and comment its content.
\begin{enumerate}
\item The traditional definition of character varieties uses an algebraic quotient, although it is not always presented that way. On the contrary, the skein algebra is given by generators and relations. These two points of view are in fact equivalent. This was previously shown by Bullock up to nilpotent elements (see \cite{bul}) and by Przytycki and Sikora in general (see \cite{ps}) using work of Brumfiel and Hilden (\cite{bh}). Indeed, the proof is in a short article of Procesi (see \cite{procesi}) and follows from the fundamental theorems of invariant theory. We explain this in Section \ref{definitions}.
 \item A theorem of K. Saito (with unpublished proof) allows to recover a representation from its character in a very general situation. We present this theorem and advertise it by applying it in different situations in Section \ref{saito}.
\begin{itemize}
\item[-] Given a character $\chi$ with values in a field $k$, in which extension of $k$ lives a representation with character $\chi$?
\item[-] Compare the tangent space of the character variety with the twisted cohomology of the group with values in the adjoint representation.
\item[-] Define tautological representations with values in the field of functions of (an irreducible component of) the character variety.
\item[-] Study points of the character variety in valuation rings, related to Culler-Shalen theory.
\end{itemize}
\item Using the tautological representation gives a convenient framework for studying global aspects of the Reidemeister torsion. We show in what sense the Reidemeister torsion may be seen as a rational volume form on the character variety of a 3-manifold with boundary and give examples in Section \ref{torsion}.
\end{enumerate}

{\bf Acknowldegements:} These working notes grew slowly and benefited from many conversations. It is my pleasure to thank L. Benard, R. Detcherry, A. Ducros, E. Falbel, L. Funar, M. Heusener, T. Q. T. Le, M. Maculan, C. Peskine, J. Porti, R. Santharoubane, M. Wolff for their help and interest. I also thank the organizers of the conference {\it Topology, Geometry and Algebra of Low-Dimensional Manifolds} in Numazu (Japan), June 2015 for welcoming me and publishing these notes. 

\section{Two definitions of character varieties}\label{definitions}
In all the article, $k$ will denote a field with characteristic $0$. Most results of Section \ref{section_saito} hold for any field with characteristic different from 2. We stayed in characteristic 0 in view of our applications. 

\subsection{Algebraic quotient}\label{git}
Let $\Gamma$ be a finitely generated group. We define its representation variety into $\sldeux$ and we denote by $\hm(\Gamma,\sldeux)$ the spectrum of the algebra 
\[ A(\Gamma)=k[X^\gamma_{i,j}, i,j\in \{1,2\}, \gamma \in \Gamma]/(\det(X^\gamma)-1,X^{\gamma\delta}-X^\gamma X^\delta\text{ with } \gamma, \delta\in \Gamma)\]
In this formula, $X^\gamma$ stands for the matrix with entries $X^\gamma_{i,j}$, $i,j\in \{1,2\}$. In particular the last equation above is a collection of 4 equations. The name representation variety is justified by the following universal property which holds for any $k$-algebra $R$: 

\[\hm_{k-\rm alg}(A(\Gamma),R)=\hm(\Gamma,\sldeux(R))\]

Let $\sldeux(k)$ act on the space $\hm(\Gamma,\sldeux)$ by conjugation. This action is algebraic as it comes from the action $(g.P)(X^\gamma)=P(g^{-1}X^\gamma g)$ where we have $g\in \sldeux(k)$ and $P\in A(\Gamma)$. 

\begin{definition}
We define the {\it character variety} of $\Gamma$ and denote by $X(\Gamma)$  the spectrum of the algebra $A(\Gamma)^{\sldeux}$ of invariants.
\end{definition}

In other words, $X(\Gamma)$ is the quotient of $\hm(\Gamma,\sldeux)$ in the sense of geometric invariant theory. Standard arguments from this theory gives the following theorem:  

\begin{theorem}
If $k$ is algebraically closed, there is a bijection between the following sets:
\begin{itemize}
\item[-] The $k$-points of $X(\Gamma)$ (or equivalently $\hm_{k-\rm alg }(A(\Gamma)^{\sldeux},k)$)
\item[-] The closed orbits of $\sldeux(k)$ acting on $\hm(\Gamma,\sldeux(k))$
\item[-] The conjugacy classes of semi-simple representations of $\Gamma$ into $\sldeux(k)$ 
\item[-] The characters of representations in $\hm(\Gamma,\sldeux(k))$. 
\end{itemize}
\end{theorem}

Recall that by character of a representation $\rho:\Gamma\to\sldeux(k)$ we mean the map $\chi_\rho:\Gamma\to k$ given by $\chi_\rho(\gamma)=\tr \rho(\gamma)$. This theorem will be made more precise in the sequel using the point of view of skein algebras. 

\begin{remark}
We point out the fact that the algebra $A(\Gamma)$ may have nilpotent elements, as the skein algebra. A big part of the literature on character varieties uses the reduction of these algebras: this is not the case of these notes. 
\end{remark}

\subsection{The skein algebra}
\begin{definition}
We define the {\it skein character variety} $X_s(\Gamma)$ as the spectrum of the algebra 

\[B(\Gamma)=k[Y_{\gamma},\gamma\in \Gamma]/(Y_1-2,Y_{\alpha\beta}+Y_{\alpha\beta^{-1}}-Y_\alpha Y_\beta\text{ with }\alpha,\beta\in \Gamma)\]
\end{definition}

One can show that $B(\Gamma)$ is a finitely generated $k$-algebra (see \cite{cs}, Proposition 1.4.1). Moreover, any representation $\rho:\Gamma \to \sldeux(k)$ gives rise to an algebra morphism $\chi_\rho: B(\Gamma)\to k$ by the formula $\chi_\rho(Y_\gamma)=\tr\rho(\gamma)$. This is a consequence of the famous trace relation: 
\[\tr(AB)+\tr(AB^{-1})=\tr(A)\tr(B)\quad\forall A,B\in \sldeux(k)\]

The character of the tautological representation $\rho:\Gamma\to \sldeux(A(\Gamma))$ defined by $\rho(\gamma)=X^\gamma$ is a map $\Phi:B(\Gamma)\to A(\Gamma)^{\sldeux}$. The following theorem as been proved by Przytyscki and Sikora, see \cite{ps}. In practice, it follows from \cite{procesi} Theorem 2.6, see also \cite{bh} and \cite{bul}.  
\begin{theorem}
For any field $k$ of characteristic 0 and any finitely generated group $\Gamma$, the map $\Phi:B(\Gamma)\to A(\Gamma)^{\sldeux}$ is an isomorphism.
\end{theorem}
The statement of Theorem 2.6 in \cite{procesi} is much more general and deals with algebras with trace satisfying the Cayley-Hamilton identity. We derive our statement from his below. 
\begin{proof}
Let $k[\Gamma]$ be the group algebra of $\Gamma$. We denote by $[\gamma]$ the generator associated to $\gamma\in \Gamma$ and set $\tr([\gamma])=[\gamma]+[\gamma^{-1}]$ that we extend by $k$-linearity. 
We define $H(\Gamma)=k[\Gamma]/I$ where $I$ is the two-sided ideal generated by the elements $\tr(x)y-y\tr(x)$, for any $x,y\in k[\Gamma]$. The trace factors to a $k$-linear endomorphism of $H(\Gamma)$. By direct computation, one checks that the following identities hold for any $x,y\in H(\Gamma)$:
\begin{itemize}
\item[(i)] $\tr(x)y=y\tr(x)$
\item[(ii)] $\tr(xy)=\tr(yx)$
\item[(iii)] $\tr(\tr(x)y)=\tr(x)\tr(y)$
\item[(iv)] $x^2-\tr(x)x+\frac{1}{2}(\tr(x)^2-\tr(x^2))=0$.
\end{itemize}
The last equation is called the Cayley-Hamilton identity of order 2. 
The map $j:H(\Gamma)\to M_2(A(\Gamma))$ defined by $j([\gamma])=X^\gamma$ is an algebra morphism preserving the trace. It is universal in the sense that for any $k$-algebra $B$ and morphism $j':H(\Gamma)\to M_2(B)$, there is a unique algebra morphism $\phi:A(\Gamma)\to B$ such that the following diagram commutes.
\[\xymatrix{ H(\Gamma)\ar[r]^j\ar[rd]^{j'} & M_2(A(\Gamma))\ar[d]^{M_2(\phi)} \\ & M_2(B)}\]
Denote by $G$ the group GL$_2(k)$. The universal property implies that if we compose $j$ with a conjugation $\pi_g$ with  $g\in G$, there is an automorphism $\phi_g$ of $A(\Gamma)$ such that $\pi_g\circ j=M_2(\phi_g)\circ j$. The action of $G$ on $A(\Gamma)$ is the one described in Subsection \ref{git}: the formula $\rho(g)=\pi_g\circ M_2(\phi_g)^{-1}$ defines an action of $G$ on $M_2(A(\Gamma))$ fixing $X^\gamma$ for all $\gamma$ and Theorem 2.6  of \cite{procesi} says the following:

\begin{theorem}[\cite{procesi}, 2.6] The map $j:H(\Gamma)\to M_2(A(\Gamma))^{{\rm GL}_2}$ is an isomorphism.
\end{theorem} 
To end the proof, we observe that $\Phi$ is the restriction of $j$ to the center. More precisely, the map $B(\Gamma)\to H(\Gamma)$ defined by $Y_\gamma\mapsto \tr([\gamma])$ is injective and its image by $j$ is the center of $M_2(A(\Gamma))^{G}$, that is $A(\Gamma)^{G}\text{Id}$. We end the proof by observing the equality $A(\Gamma)^G=A(\Gamma)^{\sldeux}$. 
\end{proof}
Thanks to this theorem, we can remove the 's' in $X_s(\Gamma)$. 

\subsection{Irreducible, reducible and central characters}
Given $\alpha,\beta\in \Gamma$, we define the following element of $B(\Gamma)$:
$$\Delta_{\alpha,\beta}=Y_\alpha^2+Y_\beta^2+Y_{\alpha\beta}^2-Y_\alpha Y_\beta Y_{\alpha\beta}-4.$$
The equation $\Delta_{\alpha,\beta}\ne0$ defines an open subset $U_{\alpha,\beta}$ of $X(\Gamma)$ 
and we set $$X^{\ir}(\Gamma)=\bigcup_{\alpha,\beta\in \Gamma} U_{\alpha,\beta}.$$

We begin our study with the following lemma, a slightly different version of Lemma 1.2.1 in \cite{cs}. We will say that a representation $\rho:\Gamma\to\sldeux(k)$ is absolutely irreducible if it is irreducible in an algebraic closure of $k$. 

\begin{lemma}\label{burnside}
Given a field $k$ and a representation $\rho:\Gamma\to\sldeux(k)$, $\rho$ is absolutely irreducible if and only if there exists $\alpha,\beta\in\Gamma$ such that $\tr(\rho(\alpha\beta\alpha^{-1}\beta^{-1}))\ne 2$.
\end{lemma}
\begin{proof}
We first recall Burnside irreducibility criterion which states that $\rho$ is absolutely irreducible if and only if $\spa\{\rho(\gamma),\gamma\in \Gamma\}=M_2(k)$.

Let $M$ be the matrix defined by $M_{ij}=\tr(\rho(\gamma_i\gamma_j))$ for $i,j\in \{1,\ldots,4\}$ where $\gamma_1=1,\gamma_2=\alpha,\gamma_3=\beta$ and $\gamma_4=\alpha\beta$. A simple computation shows that $\det M=-\chi_\rho(\Delta_{\alpha,\beta})^2$ where we have $\chi_\rho(\Delta_{\alpha,\beta})=\tr \rho(\alpha\beta\alpha^{-1}\beta^{-1})-2$. Hence if there exists $\alpha,\beta$ such that $\chi_\rho(\Delta_{\alpha,\beta})\ne 0$, then $(\rho(\gamma_i))_{i=1..4}$ is a basis of $M_2(k)$ and by Burnside criterion, $\rho$ is absolutely irreducible.

Conversely, suppose that for all $\alpha,\beta\in \Gamma$ one has $\chi_\rho(\Delta_{\alpha,\beta})=0$. Then there exists a  non-zero  subspace $F_{\alpha,\beta}$ of $k^2$ fixed by the commutator $\rho([\alpha,\beta])$.  Suppose that there exists $\gamma,\delta$ such that $F_{\alpha,\beta}\cap F_{\gamma,\delta}=\{0\}$ then in a basis made of these two lines, one can write 
$\rho([\alpha,\beta])=\begin{pmatrix} 1& x\\ 0 &1\end{pmatrix}$ and $\rho([\gamma,\delta])=\begin{pmatrix} 1& 0\\ y &1\end{pmatrix}$ with $x$ and $y$ non zero.
We compute then $\tr([[\alpha,\beta],[\gamma,\delta]])=2+(xy)^2$. The hypothesis implies that $x$ or $y$ is zero which is impossible. 
This finally implies that $\bigcap_{\alpha,\beta}F_{\alpha,\beta}\ne\{0\}$. But this subset is $\rho$-invariant which implies that $\rho$ is reducible.
\end{proof}
Consider the closed subset $X^{\cent}(\Gamma)$ of $X(\Gamma)$ defined by the equations $Y_\gamma^2=4$ for all $\gamma\in\Gamma$. A $k$-point of $X^{\cent}(\Gamma)$ has the form $\phi(Y_\gamma)=2\epsilon(\gamma)$ for some $\epsilon\in H^1(\Gamma,\Z/2\Z)$. This is the character of the central representation $\rho:\gamma\mapsto \epsilon(\gamma)\rm{id}$. 

The closed subset $X(\Gamma)\setminus X^{\ir}(\Gamma)$ is denoted by $X^{\redu}(\Gamma)$: let us describe it more precisely. Let $\hm(\Gamma,\G_m)$ be the spectrum of the algebra 
$$C(\Gamma)=k[Z_\gamma,\gamma\in \Gamma]/(Z_\gamma Z_\delta-Z_{\gamma\delta},\gamma,\delta\in \Gamma)$$
We have for any $k$-algebra $R$,  $\hm_{k-\alg}(C(\Gamma),R)\simeq \hm(\Gamma,R^\times)$, which explains the notation.

Let $\sigma$ be the automorphism of $C(\Gamma)$ defined by $\sigma(Z_\gamma)=Z_{\gamma^{-1}}$ and $\pi:B(\Gamma)\to C(\Gamma)$, the morphism defined by $\pi(Y_{\gamma})=Z_{\gamma}+Z_{\gamma^{-1}}$. 

\begin{proposition}\label{reductible}
The map $\pi:B(\Gamma)\to C(\Gamma)$ is a morphism of algebras whose image is the $\sigma$-invariant part. 
Its kernel is the radical of the ideal generated by $\Delta_{\alpha,\beta}$. 

In other words, a $k$-point $\phi$ of $X^{\redu}$ lifts to a $k$-point of $\hm(\Gamma,\G_m)/\sigma$. Hence, there exists an extension $\hat{k}$ of $k$ (at most quadratic) and a morphism $\psi:\Gamma\to \hat{k}^\times$ such that $\phi(Y_\gamma)=\psi(\gamma)+\psi(\gamma)^{-1}$. 
\end{proposition}

\begin{proof}
We prove easily that $\pi$ is a homomorphism whose image is $C(\Gamma)^\sigma$. 
Let $k$ be a field and $\phi:B(\Gamma)\to k$ a morphism satisfying $\phi(\Delta_{\alpha,\beta})=0$ for all $\alpha,\beta\in \Gamma$. We write $\phi(\alpha)=\phi(Y_\alpha)$ for short: by assumption one has $\phi([\alpha,\beta])=2$. Moreover, if $\phi(\alpha)=\phi(\beta)=2$ then $\phi(\Delta_{\alpha,\beta})=0$ implies $(\phi(\alpha\beta)-2)^2=0$ and hence $\phi(\alpha\beta)=2$. We conclude that $\phi(\alpha)=2$ for any $\alpha\in [\Gamma,\Gamma]$. Moreover, if $\alpha\in \Gamma$ and $\beta\in [\Gamma,\Gamma]$, then $\phi(\alpha\beta)$ satisfies $(\phi(\alpha)-\phi(\alpha\beta))^2=0$ and hence $\phi(\alpha\beta)=\phi(\alpha)$. This proves that $\phi$ factors through a map $\phi^{\abel}:B(\Gamma^{\abel})\to k$ where $\Gamma^{\abel}=\Gamma/[\Gamma,\Gamma]$. 
Moreover, we clearly have $C(\Gamma^{\abel})=C(\Gamma)$ and the map $B(\Gamma^{\abel})\to C(\Gamma^{\abel})^{\sigma}$ is invertible. This proves the first part of the lemma. The second part follows. 
\end{proof}
\begin{remark}
We don't know if the $\Delta_{\alpha,\beta}$s generate the kernel of $\pi$. Moreover, it is well known that the deformations of reducible representations into irreducible ones are related to the Alexander module. It would be interesting to have an explicit description of the (co-)normal bundle of $X^{\redu}(\Gamma)$ in $X(\Gamma)$. 
More precisely, define the ideal $I$ by the following exact sequence:
\[\xymatrix{0\ar[r]& I \ar[r] &B(\Gamma)\ar[r]& C(\Gamma)^\sigma\ar[r]& 0}\]
Then there should be a relation between the $C(\Gamma)$-module $I/I^2\otimes_{C(\Gamma)^\sigma} C(\Gamma)$ and the module $H^1(\Gamma,C(\Gamma))$ where the action of $\gamma\in\Gamma$ is by multiplication by $Z_\gamma$. 
\end{remark}
 
\subsection{Functorial properties}
If $\phi:\Gamma\to \Gamma'$ is a group homomorphism, there is a natural map $\phi_*:B(\Gamma)\to B(\Gamma')$ mapping $Y_\gamma$ to $Y_{\phi(\gamma)}$. There is also a natural map $\phi_*:A(\Gamma)\to A(\Gamma')$ which makes the following diagram commute:

\[\xymatrix{ B(\Gamma)\ar[d]^\Phi\ar[r]^{\phi_*} & B(\Gamma')\ar[d]^{\Phi}\\ A(\Gamma)\ar[r]^{\phi_*} & A(\Gamma')}\]

\subsubsection{Double quotients}
Let us show an example which arises when computing the fundamental group of a $3$-manifold presented by a Heegard splitting. 

\begin{proposition}
Let $\Gamma,\Gamma_1,\Gamma_2$ be three finitely generated groups and $\phi_i:\Gamma\to \Gamma_i$ be two surjective morphisms. Denote by $\Gamma'$ the amalgamated product $\Gamma'=\Gamma_1\underset{\Gamma}{\star}\Gamma_2$. 

Then $X(\Gamma')$ is isomorphic to the fiber product $X(\Gamma_1)\underset{X(\Gamma)}{\times}X(\Gamma_2)$. 
\end{proposition}
More geometrically, if we denote by $L_i$ the images of $X(\Gamma_i)$ in $X(\Gamma)$, then $X(\Gamma')$ is the (schematic) intersection $L_1\cap L_2$. 
\begin{proof}
Given any morphisms $\phi_i:\Gamma\to\Gamma_i$ and any $k$-algebra $R$, one has the following natural isomorphisms:
\begin{eqnarray*}
\hm_{k-alg}(A(\Gamma'),R)&=&\hm(\Gamma',\sldeux(R))\\
&=&\{\rho_i:\Gamma_i\to \sldeux(R), \rho_1\circ\phi_1=\rho_2\circ\phi_2\}\\
&=&\hm_{k-alg}(A(\Gamma_1)\underset{A(\Gamma)}{\otimes}A(\Gamma_2),R)
\end{eqnarray*}
Hence we have the isomorphism $A(\Gamma')=A(\Gamma_1)\underset{A(\Gamma)}{\otimes}A(\Gamma_2)$. We notice that the invariant part of this tensor product is not the tensor product of the invariant parts (think about a wedge of two circles). However, this holds when $A(\Gamma_i)$ is a quotient of $A(\Gamma)$, that is when the morphisms $\phi_i$ are surjective. 

Denoting by $I_i$ the kernel of the map $A(\Gamma)\to A(\Gamma_i)$, the result follows from the following standard fact in invariant theory: $A(\Gamma')=A(\Gamma)/(I_1+I_2)$ and $(I_1+I_2)^{\sldeux}=I_1^{\sldeux}+I_2^{\sldeux}$.

On the other hand, one shows easily that the map $B(\Gamma_1)\otimes_{B(\Gamma)}B(\Gamma_2)\to B(\Gamma')$ is an isomorphism, see for instance \cite{ps}.
\end{proof}

\subsubsection{Semi-direct products}\label{semi}
The following situation appears for fundamental groups of 3-manifolds which fiber over the circle.
Let $\Gamma$ be a finitely generated group and $\phi\in\operatorname{Aut}(\Gamma)$ an automorphism. Set $\Gamma'=\Gamma\rtimes_\phi\Z$ so that we have the following (split) exact sequence. Let $t=s(1)$.
\[\xymatrix{0\ar[r] &\Gamma\ar[r]^{\alpha}&\Gamma'\ar[r]&\Z\ar[r]\ar@/_1pc/[l]^s&0}\]
\begin{lemma}
Suppose that $k$ is algebraically closed and consider only closed points of the character varieties.
Then the map $\alpha^*:X(\Gamma')\to X(\Gamma)$ corestricted to $X^{\irr}(\Gamma)$ is a $\Z/2\Z$ principal bundle over $F$, where $F=\{x\in X^{\irr}(\Gamma), \phi^*x=x\}$.
\end{lemma}
\begin{proof}
Let $\rho':\Gamma'\to\sldeux(K)$ be a representation which restricts to an irreducible representation $\rho:\Gamma\to\sldeux(K)$. 
This representation satisfies $\rho'(t)\rho(\gamma)\rho'(t)^{-1}=\rho(\phi(\gamma))$ hence the character $\chi_\rho$ is fixed by $\phi^*$. Suppose we have another representation $\rho':\Gamma'\to\sldeux(K)$ with the same restriction, then $\rho''(t)\rho'(t)^{-1}$ commutes with the image of $\rho$. Hence $\rho''(t)=\pm \rho'(t)$. There are precisely two preimages for $\rho$ and the lemma is proved. 
\end{proof}

\section{Applications of Saito's Theorem}\label{section_saito}

We know from the isomorphism between $X_s(\Gamma)$ and $X(\Gamma)$ that if $k$ is algebraically closed, any $k$-point of $X_s(\Gamma)$ is the character of a representation $\rho:\Gamma\to \sldeux(k)$. However this fact is non-trivial to prove directly: the following theorem of \cite{saito} allows it and has further applications.
\begin{theorem}\label{saito}
Let $R$ be a $k$-algebra and $\phi:B(\Gamma)\to R$ be a morphism of $k$-algebras. Suppose that there exists $\alpha,\beta\in \Gamma$ such that $\phi(\Delta_{\alpha,\beta})$ is invertible and $A,B\in \sldeux(R)$ such that $\tr A=\phi(Y_\alpha),\tr(B)=\phi(Y_\beta),\tr(AB)=\phi(Y_{\alpha\beta})$. 
Then, there is a unique representation $\rho:\Gamma\to \sldeux(R)$ such that $\rho(\alpha)=A,\rho(\beta)=B$ and for all $\gamma\in \Gamma, \tr(\rho(\gamma))=\phi(Y_\gamma)$. 
\end{theorem}
\begin{proof}(Sketch)
Define $\gamma_i$ and $M$ as in Lemma \ref{burnside} . As $\det M=-\phi(\Delta_{\alpha,\beta})^2$, the matrix $M$ is invertible over $R$. Given any $\gamma\in \Gamma$, we denote by $C_\gamma$ the vector of coefficients of $\rho(\gamma)$ in the basis $(\rho(\gamma_i))_{i=1..4}$ and by $T_\gamma$ the vector $(\phi(Y_{\gamma\gamma_i}))_{i=1..4}$. We then have $M C_\gamma = T_\gamma$ which we solve by setting $C_{\gamma}=M^{-1}T_{\gamma}$ and define this way $\rho(\gamma)$. It remains to show that this actually defines a representation in $\sldeux$ whose traces are prescribed by $\phi$. This is a non-trivial consequence of the trace relations, we refer to \cite{saito} for a proof.  \end{proof}

\subsection{Points of $X(\Gamma)$ over arbitrary fields}
\begin{definition}
Let $k$ be a field and $\rho:\Gamma\to \sldeux(\hat{k})$ be a representation where $\hat{k}$ is an extension of $k$ and such that $\tr \rho(\gamma)\in k$ for all $\gamma\in \Gamma$. We set 
\[M(\rho)=\spa_k\{\rho(\gamma),\gamma\in \Gamma\}.\]
It is a sub-k-algebra of $M_2(\hat{k})$. We will say that two representations $\rho_i:\Gamma\to\sldeux(\hat{k}_i)$ where $i=1,2$ are equivalent if there is an isomorphism of $k$-algebras $\sigma: M(\rho_1)\to M(\rho_2)$  such that $\rho_2=\sigma\circ \rho_1$.
\end{definition}
We have the following proposition:
\begin{proposition}\label{propuniv}
Let $k$ be a field. There is a functorial isomorphism between $k$-points of $X^{\irr}(\Gamma)$ and equivalence classes of absolutely irreducible representations $\rho:\Gamma\to \sldeux(\hat{k})$ where $\hat{k}$ is a finite extension of $k$ and for all $\gamma$ in $\Gamma$, $\tr(\rho(\gamma))\in k$. 
\end{proposition}
\begin{proof}[Proof of Proposition \ref{propuniv}.]
Given $\rho:\Gamma\to\sldeux(\hat{k})$ as in the statement, we define a morphism $\phi\in \hm_{\alg}(B(\Gamma),k)$ by setting $\phi(Y_\gamma)=\tr(\rho(\gamma))$. By Lemma \ref{burnside}, there exists $\alpha,\beta\in \Gamma$ such that $\chi_\rho(\Delta_{\alpha,\beta})\ne 0$. Hence, $\phi$ indeed corresponds to a $k$-point of $X^{\ir}(\Gamma)$. 

On the other hand, a $k$-point of $X^{\ir}(\Gamma)$ corresponds to a morphism $\phi$ in $\hm_{\alg}(B(\Gamma),k)$ such that there exists $\alpha,\beta\in \Gamma$ with $\phi(\Delta_{\alpha,\beta})\ne 0$. 

We define the matrices $A=\begin{pmatrix} \phi(Y_\alpha) & -1 \\ 1 & 0\end{pmatrix}$ and $B=\begin{pmatrix} 0 & -1/u \\ u & \phi(Y_\beta)\end{pmatrix}$ where $u$ belongs to an extension $\hat{k}$ of $k$ such that the equation $u^2+u\phi(Y_{\alpha\beta})+1=0$ holds ($\hat{k}$ can be chosen at most quadratic).  By Theorem \ref{saito}, there is a unique representation $\rho:\Gamma\to \sldeux(\hat{k})$ such that $\rho(\alpha)=A,\rho(\beta)=B$ and $\tr(\rho(\gamma))=\phi(Y_\gamma)$ for all $\gamma\in \Gamma$.

In order to show that this representation does not depend on the choice of $\alpha$ and $\beta$, we observe that if we make an other choice for $A',B' \in M_2(\hat{k}')$ which defines a representation $\rho':\Gamma\to\sldeux(\hat{k}')$, the map defined by $ \sigma(A)=A'$ and $\sigma(B)=B'$  extends to a $k$-isomorphism from $M(\rho)$ to $M(\rho')$ showing that the two representations are equivalent. 
\end{proof}

The algebra $M(\rho)$ associated to an absolutely irreducible representation $\rho$ is simple and central and hence defines a class in the Brauer group $Br(k)$.  Moreover, $\dim_k M(\rho)=4$. It follows that $M(\rho)$ is a quaternion algebra and in particular, its square is 0 in $Br(K)$. The set of $k$-points of $X^{\ir}(\Gamma)$ has then a partition into Brauer classes. We get finally the following proposition:
\begin{proposition}
Let $k$ be a field with $Br(k)=0$ (this occurs for instance if $k$ is algebraically closed or of transcendence degree one over an algebraically closed field). Then $k$-points in $X^{\ir}(\Gamma)$ correspond bijectively to conjugacy classes of absolutely irreducible representations $\rho:\Gamma\to\sldeux(k)$.
\end{proposition}
\begin{proof}
An irreducible point $\phi:B(\Gamma)\to k$ corresponds to the character of a representation $\rho:\Gamma\to \sldeux(\hat{k})$ for some extension $\hat{k}$ of $k$. As the Brauer group is trivial, there is an algebra isomorphism $\sigma:M_2(k)\to M(\rho)$. Tensoring by $\hat{k}$, we find that $\sigma\otimes 1:M_2(\hat{k})\to M(\rho)\otimes\hat{k}$ is a $\hat{k}$-linear automorphism of $M_2(\hat{k})$ and hence a conjugation by some $g\in \mathrm{GL}_2(\hat{k})$ by Skolem-Noether theorem. Finally, the representation $g\rho g^{-1}:\Gamma\to \sldeux(\hat{k})$ takes its values in $k$. As two such representation $\rho$ and $\rho'$ are equivalent, it means that there is an automorphism $\sigma$ of $M_2(k)$ such that $\rho'=\sigma\circ \rho$ and this automorphism is a conjugation. 
\end{proof}

\begin{example}
Let $F_2$ be the free group on the generators $\alpha$ and $\beta$. 
It is well-known that the map $\Psi:\Q[x,y,z]\to B(F_2)$ defined by $\Psi(x)=Y_\alpha$, $\Psi(y)=Y_\beta$ and $\Psi(z)=Y_{\alpha\beta}$ is an isomorphism. Hence, one has a rational point of $X(F_2)$ by sending $x,y,z$ to $1$. However, there is no representation $\rho:F_2\to \sldeux(\Q)$ with this character. However, one can find a solution in some (non uniquely defined) quadratic extension of $\Q$. 
\end{example}

\subsection{Tangent spaces}
Let $k$ be a field and $\rho:\Gamma\to\sldeux(k)$ be an absolutely irreducible representation. The character $\chi_\rho$ may be seen as a $k$-point of $X^{\ir}(\Gamma)$. We derive from Saito's Theorem a simple proof of the following well-known result:

\begin{proposition}\label{tangent}
There is a natural isomorphism $T_{\chi_\rho}X_k(\Gamma)\simeq H^1(\Gamma,\ad_\rho)$ where $\ad_\rho$ is the adjoint representation of $\Gamma$ on $\lie(k)$ induced by $\rho$.  
\end{proposition}
\begin{proof}
It is well-known that a tangent vector at $\chi_\rho$ corresponds to an algebra morphism $\phi_\epsilon:B(\Gamma)\to k[\epsilon]/(\epsilon^2)$ such that $\phi_0=\chi_\rho$.

As $\rho$ is absolutely irreducible, there exists $\alpha,\beta\in \Gamma$ such that $\phi(\Delta_{\alpha,\beta})\ne 0$. Consider the map $\pi:\sldeux(k)\times\sldeux(k)\to k^3$ defined by $\pi(A,B)=(\tr(A),\tr(B),\tr(AB))$. Lemma \ref{submersion} below shows that this map is a submersion precisely on the preimage of the open set of $k^3$ defined by $\Delta_{\alpha,\beta}\ne 0$. This proves that we can find two matrices $A_\epsilon,B_\epsilon\in \sldeux(k[\epsilon]/(\epsilon^2))$ such that $\tr(A_\epsilon)=\phi_\epsilon(Y_\alpha),\tr(B_\epsilon)=\phi_\epsilon(Y_\beta)$ and $\tr(A_\epsilon B_\epsilon)=\phi_\epsilon(Y_{\alpha\beta})$. By Theorem \ref{saito}, there is a unique representation $\rho_\epsilon:\Gamma\to \sldeux(k[\epsilon]/(\epsilon^2))$ such that $\rho_\epsilon(\alpha)=A_\epsilon$, $\rho_\epsilon(\beta)=B_\epsilon$ and $\chi_{\rho_\epsilon}=\phi_\epsilon$.

We define the cocyle $\psi\in Z^1(\Gamma,\ad_{\rho})$ by the formula 
\begin{equation}\label{defcocycle}
\psi(\gamma)=\frac{d}{d\epsilon}\Big|_{\epsilon=0}\rho(\gamma)^{-1}\rho_\epsilon(\gamma).
\end{equation}
The map $\phi_\epsilon\mapsto \psi$ is a linear map $T_{\chi_\rho}X_k(\Gamma)\to H^1(\Gamma,\ad_\rho)$. We construct the inverse map by sending the cocycle $\psi$ to the character of the representation $\rho_\epsilon(\gamma)=\rho(\gamma)(1+\epsilon\psi(\gamma))$. 

If we had chosen other matrices $A'_\epsilon,B'_\epsilon$ defining a representation $\rho'_\epsilon$, then as in the proof of Proposition \ref{propuniv}, we would have found an automorphism $\sigma$ of $M_2(k[\epsilon](\epsilon^2))$ such that $\rho'=\sigma\circ\rho$. Moreover $\rho_\epsilon$ and $\rho'_\epsilon$ coincide up to first order so that one can write $\sigma=\text{Id}+\epsilon D$ for some derivation $D$ of $M_2(k)$. All such derivations have the form $D(X)=[\xi,X]$ for some $\xi\in \lie(k)$. A computation shows that $\psi'=\psi+d\xi$ and the map $\phi_\epsilon\to \psi$ is well-defined. Reciprocally, changing $\psi$ to $\psi+d\xi$ amounts in conjugating $\rho_\epsilon$ by $e^{\epsilon \xi}$, which does not change the character, hence the theorem is proved. 
\end{proof}

\begin{lemma}\label{submersion}
The map $\pi:SL_2(k)^2\to k^3$ defined by the formula \[\pi(A,B)=(\tr(A),\tr(B),\tr(AB))\] is a submersion at any pair $(A,B)$ such that $\tr [A,B]\ne 2$. 
\end{lemma}
\begin{proof}
Let $\xi,\eta$ be in $\lie(k)$ corresponding to a tangent vector $(A\xi,B\eta) \in T_{(A,B)}\sldeux(k)^2$. If $D_{(A,B)}\pi$ is not surjective, there exists $u,v,w\in k$ such that $u\tr(A\xi)+v\tr(B\eta)+w\tr(A\xi B+AB\eta)=0$ for all $\xi,\eta\in \lie(k)$. This is possible if and only if there exists $\lambda,\mu\in k$ such that $uA+wBA=\lambda\text{Id}$ and $vB+wAB=\mu\text{Id}$. Hence, there is a linear relation in the family $(\text{Id},A,B,AB)$, which is equivalent to the equality $\tr([A,B])=2$, see Lemma \ref{burnside}.
\end{proof}

One has to be careful that the isomorphism between $T_{\chi_\rho}X(\Gamma)$ and $H^1(\Gamma,\ad_\rho)$ does not longer hold if $\rho$ is not absolutely irreducible. However there is still a natural map $\Xi:H^1(\Gamma,\ad_\rho)\to T_{\chi_\rho}X(\Gamma)$ sending $\psi$ to the character of $\rho(1+\epsilon \psi)$. Luckily, the skein interpretation of the character variety still allows to compute the tangent space at these points. Let us consider an extreme case, that is the tangent space at the character of the trivial representation. One has $H^1(\Gamma,\ad_\rho)=H^1(\Gamma,k)\otimes \lie(k)$ and the map $\Xi$ vanishes.

\begin{proposition}
For any finitely generated group $\Gamma$ and field $k$, the tangent space $T_{1}X(\Gamma)$ at the character of the trivial representation is isomorphic to the space of functions $\psi:\Gamma\to k$ satisfying the following quadratic functional equation (see \cite{quadratic}):
\[\forall \gamma,\delta\in \Gamma, \psi(\gamma\delta)+\psi(\gamma^{-1}\delta)=2\psi(\gamma)+2\psi(\delta)\]
\end{proposition}

\begin{proof}
It is sufficient to write down the conditions for $\phi_\epsilon:Y_\gamma\mapsto 2+\epsilon\psi(\gamma)$ to be an algebra morphism from $B(\Gamma)$ to $k[\epsilon]/(\epsilon^2)$. 
\end{proof}
Observe that the map $S^2 H^1(\Gamma,k)\to T_{1}X(\Gamma)$ mapping $\lambda\mu$ to the map $\psi:\gamma\mapsto \lambda(\gamma)\mu(\gamma)$ is injective but not necessarily surjective. 

\subsection{Tautological representations}
Let $\Gamma$ be a finitely generated group and $k$ be a field. Consider an irreducible component $Y$ of $X(\Gamma)$. We will say that $Y$ is of irreducible type if it contains the character of an absolutely irreducible representation. More formally, this component corresponds to a minimal prime $\mathfrak{p}\subset B(\Gamma)$: we will write $k[Y]=B(\Gamma)/\mathfrak{p}$. The component $Y$ will be of irreducible type if and only if the tautological character $\chi_Y:B(\Gamma)\to k[Y]$ is irreducible. 

From Proposition \ref{propuniv}, there exists an extension $K$ (at most quadratic) of the field $k(Y)$ and a representation $\rho_Y:\Gamma\to \sldeux(K)$ such that $\chi_{\rho_Y}=\chi_Y$. We will call such a representation a tautological representation. 

\begin{remark}The Brauer group of a transcendance degree one extension of an algebraically closed field vanishes. Hence, if $k$ is algebraically closed and $\dim(Y)=1$, one can choose $K=\C(Y)$. This case will appear very often if $\Gamma$ is the fundamental group of a knot complement. 
\end{remark}
\begin{remark} If the component $Y$ is of reducible type, the tautological representation still exists and can be constructed directly. Moreover, it will be convenient to replace $k(Y)$ by an extension for more natural formulas. For instance, if $\Gamma^{\ab}=\Z$ and $\phi:\Gamma\to \Z$ is the abelianization morphism, then one component of $X(\Gamma)$ will be of reducible type. Moreover, it will be isomorphic to $\mathbb{A}^1$ via the map $\phi_*:B(\Gamma)\to B(\Z)\simeq k[Y_1]$. 
A tautological representation will be given by 
\[\rho_{Y}:\Gamma\to \sldeux(k(t)),\quad \rho_{Y}(\gamma)=\begin{pmatrix} t^{\phi(\gamma)}&0\\0&t^{-\phi(\gamma)}\end{pmatrix}\]
and the extension $k(Y_1)\subset k(t)$ corresponds to the substitution $Y_1=t+t^{-1}$. 
\end{remark}

\subsection{Valuation rings and Culler-Shalen theory}
We give in this section an example of result of Culler-Shalen theory (though not explicitly stated in their articles). The proof will follow standard lines for what concerns Culler-Shalen theory. We add it in order to convince the reader that the techniques of the preceding section are well-adapted to these questions. 

Let $M$ be a 3-manifold by which we mean a compact oriented and connected topological 3-manifold. An oriented and  connected surface $S\subset  M$ is said {\it incompressible} if
\begin{itemize}
\item[-] $S$ is not a 2-sphere,
\item[-] the map $\pi_1(S)\to \pi_1(M)$ induced by the inclusion is injective,
\item[-] $S$ is not boundary parallel.
\end{itemize}
 We will say that $M$ is small if it does not contain any incompressible surface without boundary.

Given a topological space $Y$ with finitely many connected components $(Y_i)_i\in I$ each of them having a finitely generated fundamental group, we set $X(Y)=\prod\limits_{i\in I} X(\pi_1(Y_i))$.  

\begin{proposition}
Let $M$ be a small 3-manifold. Then the restriction map $r:X(M)\to X(\partial M)$ is proper.
\end{proposition}
\begin{proof}
Fix a base point $y$ in $M$ and let $I=\pi_0(\partial M)$. We choose a base point $y_i$ for each component $\partial_i M$ of $\partial M$. Set $\Gamma=\pi_1(M,y)$ and $\Gamma_i=\pi_1(\partial_i M,y_i)$.
As $X(\partial M)=\prod_{i\in I} X(\Gamma_i)$, it is sufficient to define the map $r_i:X(\Gamma)\to X(\Gamma_i)$ for any $i\in I$.  This map is the one induced by the inclusion $\partial_iM\subset M$ on fundamental groups.  

By the valuative criterion of properness (see Thm 4.7 in \cite{hartshorne}), $r$ is proper if and only if for any discrete valuation $k$-algebra $R$ with fraction field $K$ fitting in the following diagram there is a morphism $\spec(R)\to X(M)$ which makes the diagram commute. 
\begin{equation}\label{cvp}\xymatrix{
\spec(K)\ar[r]^\phi\ar[d]& X(M)\ar[d]^r\\
\spec(R)\ar[r]\ar@{.>}[ru]&X(\partial M)
}
\end{equation}
We will show that if such a morphism does not exist, then there is an incompressible surface by standard Culler-Shalen arguments. 
First, by the preceding section and remarking that $K$ has transcendence degree 1 over $k$, the morphism $\phi$ corresponds to a representation $\rho:\Gamma\to \sldeux(K)$. 
Let $T$ be the Bass-Serre tree on which $\sldeux(K)$ acts. Through $\rho$, the group $\Gamma$ acts on $T$. If this action is trivial - i.e. it fixes a vertex of $T$ - the representation $\rho$ is conjugate to a representation in $\sldeux(R)$. For any $\gamma\in \Gamma$, we have $\tr \rho(\gamma)\in R$. This shows that the map $B(\Gamma)\to R$ defined by $Y_{\gamma}\mapsto \tr\rho(\gamma)$ corresponds to a map $\spec(R)\to X(M)$ which yields a contradiction. 
Hence, the action is non-trivial and is dual to a non-empty essential surface $\Sigma\subset M$. 

For all $i\in I$, denote by $\rho_i:\Gamma_i\to \sldeux(K)$ the representation induced by $\rho$. From the diagram \eqref{cvp}, we get that for all $\gamma\in \Gamma_i$, $\tr\rho_i(\gamma)\in R$. 

Let us prove that $\gamma$ fixes a vertex in $T$. If $\rho_i(\gamma)=\pm 1$, then it is trivial. In the opposite case, one can find a vector $v\in K^2$ such that $v$ and $\rho_i(\gamma)v$ form a basis. The matrix of $\rho_i(\gamma)$ in that basis is $\begin{pmatrix} 0 & -1\\ 1& \tr(\rho_i(\gamma))\end{pmatrix}$. This proves that $\rho_i(\gamma)$ is conjugate to an element in $\sldeux(R)$ and hence fixes a vertex in $T$. A standard lemma (Corollaire 3 p90 in  \cite{serre}) shows that if every $\gamma\in \Gamma_i$ fixes a vertex in $T$, then $\Gamma_i$ itself fixes a vertex in $T$. 

The relative construction of dual surfaces shows that one can find an essential surface $\Sigma$ dual to the action without boundary, see Corollary 6.0.1 in \cite{shalen}. Any component of $\Sigma$ is an incompressible surface. This is not possible as $M$ is small. 
\end{proof}

\section{The Reidemeister torsion as a rational volume form}\label{torsion}

In this section, we construct the Reidemeister torsion of a 3-manifold with boundary on an irreducible component $Y\subset X(\Gamma)$. It will be a rational volume form on $Y$, that is an element of $\Lambda^d \Omega^1_{k(Y)/k}=\Omega^d_{k(Y)/k}$ where $d=\dim Y$. Some assumptions are necessary for this construction to work, and these are given in Subsection \ref{coho}. We will end this section with examples.

\subsection{The adjoint representation}
Let $\Gamma$ be a finitely generated group, $k$ a field and $\phi:B(\Gamma)\to k$ an irreducible character. Then, we now from Proposition \ref{propuniv} that there exists an extension $\hat{k}$ of $k$ and a representation $\rho:\Gamma\to\sldeux(\hat{k})$ such that $\chi_\rho=\phi$. Moreover, the sub-algebra $M(\rho)=\spa\{\rho(\gamma),\gamma\in\Gamma\}$ satisfies $M(\rho)\otimes_k\hat{k}\simeq M_2(\hat{k})$ and depends only on $\phi$ up to isomorphism. 

Recall the splitting $M_2(\hat{k})=\hat{k}\text{Id}\oplus \lie(\hat{k})$ and for any $X\in M_2(\hat{k})$, denote by $X_0=X-\frac 1 2 \tr(X)\text{Id}$ the projection on the second factor. Then set $M(\rho)_0=\spa_k\{\rho(\gamma)_0,\gamma\in \Gamma\}$. It is a 3-dimensional $k$-vector space on which $\Gamma$ acts by conjugation. We will denote by $\ad_\phi$ this $\Gamma$-module and call it the adjoint representation. 
The relation between $\ad_\phi$ and $\ad_\rho$ is the following: \[\ad_\rho=\ad_\phi\otimes_k \hat{k}.\] 
The representation $\ad_\phi$ is more natural as it depends only on the character. 

The Killing form is an invariant non-degenerate pairing on $M(\rho)_0$ and there is an invariant volume form $\epsilon \in \Lambda^3 \left(M(\rho)_0\right)^*$ given by $\epsilon(\zeta,\eta,\theta)=\tr(\zeta[\eta,\theta])$.

\begin{proposition}\label{differ}
Let $Y$ be an irreducible component of $B(\Gamma)$ of irreducible type corresponding to a minimal prime ideal $\mathfrak{p}$. Let $\ad_Y$ be the adjoint representation associated to the tautological character $\chi_Y$. There is an isomorphism between $\Omega^1_{B(\Gamma)/k}\otimes_{B(\Gamma)}k(Y)$ and $H_1(\Gamma,\ad_Y)$. This provides the following exact sequence:
\[ \mathfrak{p}/ \mathfrak{p}^2\otimes_{k[Y]}k(Y)\to H_1(\Gamma,\ad_Y)\to \Omega^1_{k(Y)/k}\to 0.\]
\end{proposition}
\begin{proof}
This follows by duality from Proposition \ref{tangent} with the twist that $\ad_Y$ is not exactly $\ad_{\rho_Y}$ which takes values in $\sldeux(\hat{K})$ for some extension $\hat{K}$ of $k(Y)$. Instead of adapting the previous proof, we describe directly this dual picture.

Recall that $H_1(\Gamma,\ad_Y)$ is the first homology of a complex $C_*(\Gamma,\ad_Y)$ where $C_k(\Gamma,\ad_Y)$ is generated by elements of the form $\xi \otimes[\gamma_1,\ldots,\gamma_k]$ with $\xi\in \ad_Y$ and $\gamma_1,\ldots,\gamma_k\in \Gamma$. 

We have the formulas $\partial \xi\otimes [\gamma]=\rho(\gamma)^{-1}\xi\rho(\gamma)-\xi$ and $\partial \xi\otimes[\gamma,\delta]=\rho(\gamma)^{-1}\xi\rho(\gamma)\otimes [\delta]-\xi\otimes[\gamma\delta]+\xi\otimes[\gamma]$. This allows to check directly that the map $dY_\gamma\mapsto \rho(\gamma)_0\otimes[\gamma]$ induces a well-defined map of $B(\Gamma)$-modules $\Omega^1_{B(\Gamma)/k}\to H_1(\Gamma,\ad_Y)$ and hence a map 
\[\Psi: \Omega^1_{B(\Gamma)/k}\otimes k(Y)\to H_1(\Gamma,\ad_Y).\]
%By universal coefficients (working over the field $k(Y)$), there is an isomorphism $H^1(\Gamma,\ad_Y)\simeq \hm(H_1(\Gamma,\ad_Y),k(Y))$. To show that the map si surjective, we pick $[z]\in H^1(\Gamma,\ad_Y)$ such that $\langle [z], \rho(\gamma)_0\otimes\gamma\rangle=0$ for all $\gamma$. This amounts in saying $\tr(\rho(\gamma)z(\gamma))=0$. Putting $\rho_\epsilon=\rho(1+\epsilon z)$, we get a representation $\Gamma\to \sldeux(\hat{k}[\epsilon]/(\epsilon^2))$ satisfying $\chi_\rho=\chi_{\rho_\epsilon}$. 

Reciprocally, let $\Lambda=k(Y)\oplus \epsilon \left(\Omega^1_{B(\Gamma)/k}\otimes k(Y)\right)$ be the $k(Y)$-algebra with $\epsilon^2=0$. The map $\phi:B(\Gamma)\to \Lambda$ given by $Y_\gamma\to Y_\gamma+\epsilon dY_\gamma$ is an algebra morphism. By Lemma \ref{submersion}, one can find $A_\epsilon,B_\epsilon\in \sldeux(\Lambda)$ whose traces are prescribed and by Saito's theorem (using the irreducibility of $\chi_Y)$, one finds a representation $\rho_\epsilon:\Gamma\to \sldeux(\Lambda)$ satisfying $\chi_{\rho_\epsilon}=\phi$. A simple check shows that the map $\xi\otimes \gamma\mapsto \frac{d}{d\epsilon}\tr(\xi\rho(\gamma)^{-1}\rho_\epsilon(\gamma))$ induces precisely the inverse of $\Psi$. 
The exact sequence of the proposition follows from standard facts of commutative algebra. 
\end{proof}

\subsection{Characters of 3-manifolds}\label{coho}
%Suppose $k$ is algebraically closed. 
Let $M$ be a connected 3-manifold whose boundary is non empty and has the following decomposition into connected components: $\partial M=\bigcup_{i\in I} \partial_i M$. Let $g_i$ be the genus of $\partial_i M$ : we suppose that $g_i>0$ for all $i\in I$ and set $d=\sum_i\max(1,3g_i-3)$.

Denote by $\Gamma$ the fundamental group of $M$ and let $Y$ be an irreducible component of $X(\Gamma)$. We denote by $\chi_Y:B(\Gamma)\to k[Y]$ the tautological character and by $\ad_Y$ the adjoint representation. The following technical lemma will be useful in the sequel.

\begin{lemma}\label{localisation}
Let $\alpha,\beta$ be two elements of $\Gamma$ such that $\Delta_{\alpha,\beta}$ is non zero in $k[Y]$. Then, there is a $k(Y)$-basis of $\ad_Y$ such that the adjoint representation has coefficients in $k[Y][\frac{1}{\Delta_{\alpha,\beta}}]$. 
\end{lemma}
\begin{proof}
Set $R=k[Y][\Delta_{\alpha,\beta}^{-1}]$ and consider the composition $\chi_Y:B(\Gamma)\to k[Y]\to R$. One can find an extension $\hat{R}$ with $A,B\in \sldeux(\hat{R})$ and by Saito's theorem, a unique representation $\rho_{\alpha,\beta}:\Gamma\to \sldeux(\hat{R})$ with character $\chi_Y$ satisfying $\rho(\alpha)=A$ and $\rho(\beta)=B$. By construction, $\rho(\gamma)$ lies in $\spa_{R}\{1,A,B,AB\}$. The basis $A_0,B_0,(AB)_0$ of $M(\rho_{\alpha,\beta})_0$ satisfies the assumption of the lemma and one has $M(\rho)_0\simeq M(\rho_{\alpha,\beta})_0\simeq\ad_Y$. This proves the lemma.
\end{proof}

\subsubsection{Computing $H^1$}

\begin{definition}
Recall that we say that a representation $\rho:\Gamma\to \sldeux(k)$ belongs to $Y$ if its character $\chi_\rho:B(\Gamma)\to k$ factorizes through $k[Y]$. Equivalently, we will say that $Y$ contains the representation $\rho$. For instance, a representation is of irreducible type if it contains an irreducible character. 

A representation $\rho:\Gamma\to\sldeux(k)$ is said {\it regular} if $H^1(\Gamma,\ad_\rho)$ has dimension $d$. 
An irreducible component $Y$ of $X(\Gamma)$ will be said regular if it contains a regular representation. 
\end{definition}

\begin{proposition}\label{inegpoinc}
An irreducible component $Y$ of $X(\Gamma)$ of irreducible type is regular if and only if $\dim H^1(\Gamma,\ad_Y)=d$. 
\end{proposition}
\begin{proof}
We start with a standard argument of Poincar\'e duality, taken from \cite{hk}. Let $\rho:\Gamma\to\sldeux(k)$ be any representation and $k$ be any field. 
From Poincar\'e duality and using the trace to identify $\ad_\rho$ and $\ad_\rho^*$, we get for any integer $l$ an isomorphism of $k$-vector spaces:
$$PD:H^l(M,\ad_\rho)\simeq H^{3-l}(M,\partial M,\ad_{\rho})^*$$
The exact sequence of the pair $(M,\partial M)$ and Poincar\'e duality together give the following commutative diagram:
$$\xymatrix{
H^1(M,\ad_\rho)\ar[r]^\alpha\ar[d]^{PD} & H^1(\partial M,\ad_\rho) \ar[r]^\beta\ar[d]^{PD} & H^2(M,\partial M,\ad_\rho)\ar[d]^{PD} \\
H^2(M,\partial M,\ad_\rho)^*\ar[r]^{\beta^*} & H^1(\partial M,\ad_\rho)^* \ar[r]^{\alpha^*} & H^1(M,\ad_\rho)^*
}$$
This gives $\rk(\alpha)=\rk(\beta)=\dim\ker\beta$, $\dim\ker\beta+\rk\beta=\dim H^1(\partial M,\ad_\rho)$ and finally $\rk(\alpha)=\frac 12  \dim H^1(\partial M,\ad_\rho)$. Hence, $\dim H^1(M,\ad_\rho)\ge \frac{1}{2} \dim H^1(\partial M, \ad_\rho)$. 

Moreover, $\chi(H^*(\partial_i M,\ad_\rho))=3\chi(\partial_iM)=6-6g_i$ and $H^0(\partial_i M,\ad_\rho)$ and $H^2(\partial_i M,\ad_\rho)$ have the same dimension. Hence, $\dim H^1(M,\ad_\rho)\ge \sum_i (\dim H^0(\partial_i M)+3g_i-3)$. 
If $g_i=1$, $\rho$ restricted to $\partial_i M$ is abelian and hence, $\dim H^0(\partial_i M,\ad_\rho)\ge 1$. This proves the following inequality whatever be $\rho$:

\begin{equation}\label{poincare}
\dim H^1(M,\ad_\rho)\ge d.
\end{equation}

Let $\rho:\Gamma\to\sldeux(k)$ be a regular representation and $\alpha,\beta\in\Gamma$ be such that $\chi_\rho(\Delta_{\alpha,\beta})\ne 0$. Then, by Lemma \ref{localisation}, one can find a $k(Y)$-basis of $\ad_Y$ such that the adjoint representation has coefficients in $R=k[Y][\frac{1}{\Delta_{\alpha,\beta}}]$. Let us call $\ad_Y^R$ the free $R$-module such that $\ad_Y=\ad_Y^R\otimes_R k(Y)$. 

Then $C^*(M,\ad_Y^R)$ is a finite complex of free $R$-modules. By standard semi-continuity arguments (See \cite{hartshorne}, Chap. 12), the function $p_i(x)=\dim_{k(x)} H^i(M,\ad^R_Y\otimes k(x))$ is upper semi-continuous on $\spec(R)$ for any $i\in \N$. Comparing the character $\chi_\rho$ with the generic point, we get the inequality $\dim_{k(Y)}H^1(M,\ad_Y)\le \dim_k H^1(M,\ad_\rho)$. 

It follows that if $Y$ contains a regular representation, then $H^1(M,\ad_\rho)$ has dimension $d$. Reciprocally, if $H^1(M,\ad_Y)$ has dimension $d$, there is an open set in $\spec(R)$, and hence some closed points $\chi_\rho$ for which $H^1(M,\ad_\rho)$ has dimension less than $d$ by upper semi-continuity, hence equal to $d$ by the inequality \eqref{poincare}.
\end{proof}

\begin{corollary}
Let $Y$ be an irreducible component of $X(\Gamma)$ which contains a regular representation and such that $\mathfrak{p}/\mathfrak{p}^2\otimes k(Y)=0$ where $\mathfrak{p}$ is the minimal prime ideal of $B(\Gamma)$ associated to $Y$. Then $Y$ has dimension $d$. 
\end{corollary}

\begin{proof}
From Proposition \ref{differ} and Proposition \ref{inegpoinc}, we get the sequence of inequalities $d=\dim H^1(M,\ad_Y)=\dim \Omega^1_{k(Y)/k}={\rm tr.deg}_{k}k(Y)$. \end{proof}

If the assumptions of the corollary are verified, we will say that $Y$ is a regular component of $X(\Gamma)$. If $M$ is a hyperbolic manifold with finite volume and $\rho:\pi_1(M)\to \sldeux(k)$ is a lift of the holonomy representation, then the component of $X(\Gamma)$ containing $\rho$ is regular as $\rho$ is regular and  $\dim Y=d$. We suppose from now on that $Y$ is a regular component of $X(\Gamma)$. 

\subsubsection{Computing $H^2$}

%In the same way, we compute the other two cohomology groups of $M$ with coefficients in $\ad_Y$ supposing that $Y$ is a regular component.
%From the equality in \eqref{poincare}, we get $\dim H^0(\partial_i M,\ad_Y)=0$ if $g_i>1$ and $1$ if $g_i=1$. In the case when $g_i=1$, we would like to find a preferred generator. 

%Let $I_0\subset I$ be the set of toric boundary components of $M$, we will say that a generator $\xi_i\in H^0(\partial_i M,\ad_Y)$ is normalized if $\tr(\xi_i^2)=2\in H^0(\partial_i M,k)$. Such a generator exists if and only if $\frac{1}{2}\tr(\xi_i^2)\in k(Y)^{\times 2}$ for any (or all) generators $\xi_i$ of $H^0(\partial_i M,\ad_Y)$. We will consider an extension $K$ of $k(Y)$ in which a collection $(\xi_i)_{i\in I_0}$ is given. This extension has degree at most $2^{|I_0|}$. 

%A more geometric point of view is the following: fix a base point $x$ in $M$ and choose arcs joining the base points of $\partial_i M$ to $x$. In that way, we get maps $\pi_1(\partial_i M)\to \Gamma$. We call extended representation a pair $(\rho,(z_i))\in \hm(\Gamma,\sldeux(k))\times\mathbb{P}^1(k)^{I_0}$ such that $\rho(\gamma_i)z_i=z_i$ for all $\gamma_i\in \pi_1(\partial_i M)$. The algebraic quotient of this space is an algebraic variety say $X^+(\Gamma)$ and the map forgetting the $z_i$'s is an algebraic map $X^+(\Gamma)\to X(\Gamma)$ which is (generically) a $2^{|I_0|}$-ramified covering. The field $K$ is the field of fractions of an irreducible component $Y^+$ of $X^+(\Gamma)$ mapping to $Y$. 
Let $I_0\subset I$ be the subset parametrizing the toric components of $\partial M$ and let $Y$ be a regular component of $X(\Gamma)$. Then, from the equality in \eqref{poincare}, we get the isomorphisms $H^0(\partial_i M,\ad_Y)=0$ if $g_i>1$ and $H^0(\partial_i M,\ad_Y)$ is one dimensional if $g_i=1$. Pick $\xi_i$ a generator of $H^0(\partial_i M)$ for $i\in I_0$.
\begin{lemma}
The map $H^2(M,\ad_Y)\to K^{I_0}$ mapping $\eta$ to the family $(\langle r_i^*\eta,\xi_i\rangle)_{i\in I_0}$ is an isomorphism, where $r_i:\partial_i M\to M$ is the inclusion map. 
\end{lemma}
\begin{proof}
This map is part of the exact sequence of the pair $(M,\partial M)$. Poincar\'e duality gives $H^2(\partial M,\ad_Y)=\bigoplus_i H^0(\partial_i M,\ad_Y)^*=K^{I_0}$ where a basis is given by evaluation on $\xi_i$. The map is surjective as the group $H^3(M,\partial M;\ad_Y)=H^0(M,\ad_Y)^*=0$ as $\ad_Y$ is an irreducible representation. From a computation of Euler characteristic, we get $\dim H^2(M,\ad_Y)=|I_0|$ and the result follows.
\end{proof}

\subsection{Reidemeister Torsion}
For a vector space $E$ of dimension $n$ over $k$, we denote by $\det(E)$ the space $\Lambda^nE$ and if $n=1$, we set $E^{-1}=E^*$. 
Given a finite complex of finite dimensional vector spaces $C^*=C^0\to\cdots C^k$, we set $\det C^*=\bigotimes_{i=0}^k (\det C^i)^{(-1)^i}$. The cohomology of $C^*$ is viewed as a complex with trivial differentials. There is a natural (Euler) isomorphism $\det C^*\simeq \det H^*$ between the determinant of a complex and the determinant of its cohomology which is well-defined up to sign. 

Picking a cellular decomposition of $M$ with $0,1$ and $2$-cells, we obtain a complex $C^*(M,\ad_Y)$ with a preferred volume element: we associate to each cell the volume element in $\Lambda^3 \ad_Y$ dual to the trilinear form $\epsilon$.

Through the isomorphism $\det C^*(M,\ad_Y)\simeq \det H^*(M,\ad_Y)$, we get an element $T(M)\in \det H^*(M,\ad_Y)=\det \Omega^1_{k(Y)/k}\otimes \det H^2(M,\ad_Y)$. Given a system of generators $\xi=(\xi_i)_{i\in I_0}$ of $H^0(\partial M,\ad_Y)$, we define 
\[T(M,\xi)=\langle T(M),\bigotimes_{i\in I_0} \xi_i\rangle\in \Omega^d_{k(Y)/k}.\]

In practice, we will evaluate it on preferred generators $\xi_i\in H^0(\partial_i M,\ad_Y)$ such that $\tr(\xi_i^2)=2$ or on elements of the form $\xi_i=\rho(\gamma_i)_0$ for some $\gamma_i\in \pi_1(\partial_i M)$.

\subsubsection{The handlebody} 

Let $M$ be handlebody of genus 2. Let $\alpha,\beta$ be the generators of $F_2=\pi_1(M)$. One has $B(\Gamma)=k[x,y,z]$ with $x=Y_\alpha,y=Y_\beta$ and $z=Y_\gamma$. Moreover a tautological representation is given by 
\[\rho(\alpha)=\begin{pmatrix} x & -1 \\ 1 & 0\end{pmatrix}, \quad \rho(\beta)=\begin{pmatrix} 0 & -u^{-1} \\ u & y\end{pmatrix}\]
with $K=k(x,y,z,u)/(u^2+uz+1)$. 
The handlebody collapses to a cellular complex with one 0-cell and two 1-cells and the twisted complex is given by $C^0=\lie(K), C^1=\lie(K)^2$ and 

\[d\xi= (\rho(\alpha)^{-1}\xi\rho(\alpha)-\xi, \rho(\beta)^{-1}\xi\rho(\beta)-\xi)\]
Let us write $T(M)=fdx\wedge dy\wedge dz$. Then, in order to compute $f$, it suffices to evaluate it on the vector fields $\partial_x,\partial_y,\partial_z$ which correspond to the twisted cocycles $\psi_x=\rho^{-1}\partial_x\rho,\psi_y=\rho^{-1}\partial_y\rho,\psi_z=\rho^{-1}\partial_z\rho$. 

Finally, $f$ is the determinant of the matrix of $d$ and the three cocycles in a basis of $\lie(K)$ of volume 1. This gives $f=2$ and $T(M)=2 dx\wedge dy\wedge dz$. 

\subsubsection{The magic manifold}
Let $p:S^3\to S^2$ be the Hopf fibration and $\Sigma\subset S^2$ be the complement of three disjoint discs. Then $M=p^{-1}(\Sigma)$ is the complement of a link in $S^3$ with 3 components called the magic manifold. Clearly, the Hopf bundle restricted to $\Sigma$ is trivial and we have indeed $M\simeq \Sigma\times S^1$.

In particular, $\pi_1(M)=F_2\times \Z$. If $\rho:\Gamma\to \sldeux(k)$ is irreducible, it has to map the generator of $\Z$ to $\pm 1$. This shows that $X(\Gamma)$ has two irreducible components $Y_+$ and $Y_-$ of irreducible type, each isomorphic to $X(F_2)=\mathbb{A}^3$. 
Let $\alpha,\beta$ be the generators of $F_2$ and $t$ the generator of $\Z$. The tautological representation $\rho:\Gamma\to \sldeux(K)$ corresponding to $Y_{\pm}$ is given by the same formula as above with in addition $\rho(t)=\pm \text{Id}$. 

The manifold $M$ collapses to the product of a wedge of two circles with a circle. This cell complex has respectively 1,3,2 cells of dimension 0,1,2.
We have 
\[d^0\xi=(\rho(\alpha)^{-1}\xi\rho(\alpha)-\xi, \rho(\beta)^{-1}\xi\rho(\beta)-\xi,0)\text{ and }\]

\[d^1(\zeta,\eta,\theta)=(\rho(\alpha)^{-1}\theta\rho(\alpha)-\theta,\rho(\beta)^{-1}\theta\rho(\beta)-\theta)\]

Without surprise, this complex is the tensor product of the complex of the handlebody with the cellular complex of the circle. In particular, it splits into two independent complexes. Its torsion will be $\frac{2}{g}dx\wedge dy\wedge dz$ where $g$ is the torsion of the acyclic complex

$$\xymatrix{
0\ar[r]&\sldeux(K)\ar[r]^{d}& \sldeux(K)^2\ar[rrr]^{\langle \cdot,\xi_\alpha\otimes\xi_\beta\otimes \xi_{\alpha\beta}\rangle}&&&K^3\ar[r]&0
}$$
Using $\xi_\alpha=\rho(\alpha)_0, \xi_\beta=\rho(\beta)_0$ and $\xi_{\alpha\beta}=\rho(\alpha\beta)_0$, we get the formula $g=4$ and hence $T(M,\xi)=\frac{1}{2}dx\wedge dy\wedge dz$. In the normalization $\tr\xi_i^2=2$, we get 
\[T(M)=\frac{dx\wedge dy\wedge dz}{2\sqrt{(x^2-4)(y^2-4)(z^2-4)}}=\frac{1}{2}\frac{du}{u}\wedge\frac{dv}{v}\wedge\frac{dw}{w}\]
 where we have set $x=u+u^{-1},y=v+v^{-1}$ and $z=w+w^{-1}$. 

\subsubsection{Manifolds fibering over the circle}
The following example is an adaptation of \cite{dub}. Let $\Sigma$ be compact oriented surface with boundary and $\phi:\Sigma\to\Sigma$ be a homeomorphism preserving the orientation and fixing the boundary pointwise. 
The suspension $M$ is defined as $M=\Sigma\times [0,1]/\!\!\sim$ where $(x,1)\sim (\phi(x),0)$. Its boundary is $\partial\Sigma\times S^1$ and its fundamental group is 
\[\pi_1(M)=\pi_1(\Sigma)\rtimes \Z\]
where the action of $\Z$ is given by the action $\phi$ on $\pi_1(M)$. More precisely, picking an element $t\in\pi_1(M)$ mapping to 1, we have $t\gamma t^{-1}=\phi_*(\gamma)$ where $\phi_*:\pi_1(\Sigma)\to\pi_1(\Sigma)$ is the map induced by $\phi$. 
 
%Let $Y$ be an irreducible component of irreducible type of $\pi_1(M)$ and $\rho_Y:\pi_1(M)\to \sldeux(K)$ be a tautological representation. 
%Let $\rho:\pi_1(M)\to \sldeux(k)$ be an irreducible representation. By the inclusion $[\pi_1(M),\pi_1(M)]\subset \pi_1(\Sigma)$, the restriction of $\rho$ to $\pi_1(\Sigma)$ is again irreducible. Reciprocally, if two representations $\rho,\rho':\pi_1(M)\to \sldeux(k)$ restrict to the same irreducible representation of $\Sigma$, then $\rho(t)\rho'(t)^{-1}$ commutes with this restriction hence $\rho'(t)=\pm \rho(t)$. 
Recall from Subsection \ref{semi} that the map $r:X(M)\to X(\Sigma)$ corestricted to $X^{\irr}(\Sigma)$ is a $\{\pm 1\}$-principal covering on its image, the (irreducible) fixed point set of $\phi^*$ acting on $X^{\irr}(\Sigma)$. 
We choose an irreducible component $Y$ of $X(M)$ of irreducible type: it maps to an irreducible component $Z=r(Y)$ of $X(\Sigma)$ of irreducible type.
Moreover, the map $Y\to Z$ is a covering of order at most 2. The adjoint representation $\ad_Y$ is clearly independent on the representation of $\pi_1(M)$ chosen, hence the Reidemeister torsion actually lives on $Z$. 

Consider the long exact sequence of the pair $(M,\Sigma)$ (Wang sequence), together with the corresponding sequence on the boundary. It gives the following diagram (we removed the coefficients $\ad_Y$ from the notation).

\[\xymatrix{
0\ar[r]\ar[d]& H^1(M)\ar[r]\ar@{(->}[d] & H^1(\Sigma)\ar[r]^\alpha\ar[d]& H^1(\Sigma) \ar[r]\ar[d]& H^2(M)\ar[r]\ar[d]^\sim& 0\\
H^0(\partial \Sigma)\ar[r]&H^1(\partial M)\ar[r]& H^1(\partial\Sigma)\ar[r]& H^1(\partial\Sigma)\ar[r]^\sim& H^2(\partial M) \ar[r]& 0
}\] 
Multiplicativity properties imply that the torsion of the first line is the element $T(M)\in \det\Omega^1_{k(Y)/k}\otimes \det H^2(M)$. Choosing generators $\xi$ of $H^0(\partial M)$, gives the preferred generator of $\det H^2(M)$ and hence the torsion $T(M,\xi)$ as a volume form on $Z$. 

It remains to interpret the map $\alpha=\phi^*-\text{id}$, which we do in the following lemma, proved in the same way as Proposition \ref{differ}.
\begin{lemma}
The natural map $\Omega^1_{B(\Sigma)/k}\otimes_{k[Z]} k(Y)\to H_1(\Sigma,\ad_Y)$ is an isomorphism. In geometric terms, $H_1(\Sigma,\ad_Y)$ is the space of rational sections of the cotangent bundle of $X(\Sigma)$, pulled-back to $Y$. 
\end{lemma}
In particular, the map $\phi^*$ is simply the derivative of the action of $\phi$ on $X(\Sigma)$. Restricted to $Z$, it is an endomorphism of the restriction of the cotangent space of $X(\Sigma)$ to $Z$. 

For any component $\gamma_i$ of $\partial \Sigma$, $i\in \pi_0(\partial\Sigma)$, we consider the generator $\xi_i=\rho_Y(\gamma_i)_0$. Plugging it into the sequence, we have the following interpretation of $T(M,\xi)$ and its normalized version

\begin{proposition}
On any irreducible component of the fixed point set of $\phi^*$ on $X(\Sigma)$ we have 
\[T(M,\xi)=\frac{1}{2}\frac{\bigwedge dY_{\gamma_i}}{\det(D\phi-1)|_{\ker dY_{\gamma_i}}}\text{ and }T(M)=\frac{1}{2}\frac{\bigwedge u_i^{-1}du_i}{\det(D\phi-1)|_{\ker du_i}}\]
where $u_i+u_i^{-1}=Y_{\gamma_i}$.
\end{proposition}
In particular, we recover the example of the magic manifold, where $\phi$ is the identity.

\end{document}